\documentclass[11pt, letterpaper]{amsart}   	
\usepackage{graphicx}				
\usepackage{amssymb}

\usepackage{latexsym,exscale,enumerate,amsfonts,amssymb,mathtools}
\usepackage{amsmath,amsthm,amsfonts,amssymb,amscd,stmaryrd,textcomp}
\usepackage[normalem]{ulem}
\usepackage{young}
\usepackage{thmtools,xcolor}
\usepackage{easybmat}

\definecolor{colormy}{rgb}{0.8,0.05,0.05}
\definecolor{mycolor}{rgb}{0.25,0.99,0.25}

\usepackage[all]{xy}
\SelectTips{cm}{}

\usepackage{tikz}
\usetikzlibrary{decorations.markings}
\usetikzlibrary{decorations.pathreplacing}
\usetikzlibrary{arrows,shapes,positioning}
\tikzstyle directed=[postaction={decorate,decoration={markings,
    mark=at position #1 with {\arrow{>}}}}]
\tikzstyle rdirected=[postaction={decorate,decoration={markings,
    mark=at position #1 with {\arrow{<}}}}]

\usepackage{hyperref}

\newcommand{\Hom}{\mathrm{Hom}}

\newcommand{\Ext}{\mathrm{Ext}}

\newcommand{\spa}{\mathrm{span }}

\def\C{{\mathbb C}}

\def\Z{{\mathbb Z}}

\def\Ind{\mathrm{Ind}}

\def\cha{\mathrm{ch}}

\theoremstyle{definition}
\newtheorem{thm}{Theorem}[section]
\newtheorem{cor}[thm]{Corollary}
\newtheorem{lem}[thm]{Lemma}
\newtheorem{prop}[thm]{Proposition}

\theoremstyle{definition}
\newtheorem{defn}[thm]{Definition}
\newtheorem{rem}[thm]{Remark}

\numberwithin{equation}{section}

\begin{document}

\title{Extensions of simple modules for quantum groups at complex roots of $1$} 

\author{Henning Haahr Andersen}
\address{Center for Quantum Mathematics (QM), Imada,
SDU, Denmark}
\email{h.haahr.andersen@gmail.com}

\date{}							




\begin{abstract}
Let $U_q$ be the quantum group corresponding to a complex simple Lie algebra $\mathfrak g$ with root system $R$. Assume the quantum parameter $q\in \C$ is a root of unity.  In this paper we study the extensions between simple modules in the category consisting of the finite dimensional modules for $U_q$. We first prove that this problem is equivalent to finding the extensions between the finitely many simple modules for the small quantum group $u_q$ in $U_q$. Then we show that the extension groups in question are determined by a finite subset with small highest weights. When the order of $q^2$ is at least the Coxeter number for $R$ we prove that the dimensions of such extension groups equal the top degree coefficients of some associated Kazhdan-Lusztig polynomial for the affine Weyl group for $R$.

We relate all this to similar (old) results for almost simple algebraic groups and their Frobenius subgroup schemes over fields of large prime characteristics.

\end{abstract}

\maketitle

\section{Introduction}
\subsection{Notation}
Let $R$ denote an irreducible finite root system. Choose a set of positive roots $R^+$ in $R$ and let $S\subset R^+$ be the set of simple roots.

Let $\mathfrak g$ denote the simple complex Lie algebra with root system $R$. Then we shall denote the category of finite dimensional $\mathfrak g$-modules by $\mathcal C_0$. This is also the category of finite dimensional modules for the enveloping algebra $U_0 =U({\mathfrak g})$, or equivalently for the complex algebraic group $G_0$ with Lie algebra $\mathfrak g$. This wellknown category is semisimple and the characters of its simple modules are given by Weyl's character formula.

Suppose $q\in \C$ is a root of unity and let $\ell$ denote the order of $q^2$. Then we shall consider the quantum group  $U_q$  obtained via Lusztig's $q$-divided power construction. 
The corresponding category of finite dimensional modules for $U_q$ is denoted $\mathcal C_q$.  We shall only consider modules of type $\bf 1$. We call $\mathcal C_q$ the quantum category.

We denote by $G$ the almost simple algebraic group over a field $k$ of characteristic $p>0$ corresponding to $R$. Then the category of finite dimensional modules for $G$ will be denoted $\mathcal C_p$ and called the modular category. 

Consider the set $X \simeq \Z^n$ of integral weights for $R$ and let $X^+$ be the set of dominant weights corresponding to our choice of positive roots $ R^+$.  Then the simple modules in the three categories $\mathcal C_0$, $\mathcal C_p$ and $\mathcal C_q$ are all parametrized by $X^+$ via highest weights. We denote the simple module in $\mathcal C_0$,  $\mathcal C_p$, respective $\mathcal C_q$ associated with $\lambda \in X^+$ by $L_0(\lambda)$, $L_p(\lambda)$, respectively $L_q(\lambda)$. 

In all three categories we denote by $[M:L]$ the composition factor multiplicity of a simple module $L$ in a general module $M$. In the semisimple category $\mathcal C_0$ this symbol is then just the number of summands in $M$ isomorphic to $L$. Moreover, in $\mathcal C_0$ the simple objects $L_0(\lambda)$ are also known as the Weyl modules for $\mathfrak g$ and as mentioned above their characters are given by the famous Weyl character formula. On the other hand, the structures of the two categories $\mathcal C_q$ and $\mathcal C_p$ are much more interesting/complicated and their simple modules and in particular the extensions between these will be the main focus in this paper. In the modular case we dealt with this issue more than 40 years ago (see \cite{A82}). In the present paper we focus on the quantum case. Of course many new results have been obtained over the last 4 decades and we take advantage of all those which are relevant in our case. This also makes it possible to simplify some of the arguments used in \cite{A82}, and we have inserted several remarks on the modular case  where this is relevant. 

We shall need the following additional notation. The small quantum group $u_q$ is the subalgebra of $U_q$ generated by those standard generators of $U_q$ which are killed by the quantum Frobenius homomorphism $Fr_q: U_q \to U_0$. Likewise the Frobenius kernel $G_1$ is the subgroup scheme of $G$ consisting of those elements, which are mapped to $1$ by the Frobenius homomorphism $Fr_p: G \to G$. 

If $M \in \mathcal C_q$ is a module on which $u_q$ acts trivially, then $M = M_1^{[q]}$ for some module $M_1 \in \mathcal C_0$. Here the notation $M_1^{[q]}$ means that the $U_q$-action on the vector space $M_1$ factors through $Fr_q$. We then also write $M_1 = M^{[-q]}$. Likewise, if $L_1 \in \mathcal C_p$ then the Frobenius twist $L = L_1^{(1)}$ of $L_1$ is the module $L$ with the same underlying vector space as $L_1$ but with $G$-action obtained by first applying $Fr_p$.  Again we write $L_1 = L^{(-1)}$. 

In the quantum case we set 
$$ X_1 = \{\lambda \in X^+ \mid \langle \lambda, \alpha^\vee \rangle \leq \ell - 1 \text { for all } \alpha \in S\}.$$
In the modular case the very same notation is used for the subset of $X^+$ where the analogous inequalities with $\ell$ replaced by $p$ hold. In the quantum case we write for an arbitrary $\lambda \in X$
$$ \lambda = \lambda^0 + \ell \lambda^1 \text { with } \lambda^0 \in X_1, \lambda ^1 \in X$$ 
and in the modular case
$$ \lambda = \lambda^0 + p \lambda^1 \text { with } \lambda^0 \in X_1, \lambda ^1 \in X.$$ 
It will be clear from the context which of the two ``decompositions'' of $\lambda$ this notation refers to.

\subsection{The extension problem}
Once the above mentioned classifications of simple modules in $\mathcal C_q$, respectively $\mathcal C_p$, were obtained, the main fundamental problem in the representation theories for $U_q$ and $G$ was to determine their irreducible characters, i.e. $\cha L_q(\lambda)$ and $\cha L_p(\lambda)$ for all $\lambda \in X^+$. Here for a module $M$ in either  $\mathcal C_q$ or $\mathcal C_p$  we use the notation $\cha M$ for the character of $M$, i.e. 
$$ \cha M = \sum_{\mu \in X} \dim M_\mu \; e^\mu \in \Z[X],$$
where $M_\mu$ denotes the weight subspace in $M$ corresponding to $\mu$. In $\mathcal C_q$ this problem was solved in \cite{KT1} and {\cite{KT2}, and for $p$ very large the corresponding problem for $\mathcal C_p$  was dealt with in \cite{AJS}. In both cases the solution was given in terms of the Kazhdan-Lusztig polynomials associated with the affine Weyl group for $R$ in accordance with conjectures by Lusztig, see [Lu89] and [Lu79].  (For the remaining primes the characters of the simple modules in $\mathcal C_p$ have been found much more recently. It is given in terms of the socalled $p$-Kazhdan-Lusztig polynomials, see \cite{RW1}, \cite{RW2} and \cite{W}.

The next fundamental problem is then to describe how to build modules in $\mathcal C_q$ and $\mathcal C_p$ from their composition factors. A key ingredient in doing so is to determine the extensions in $\mathcal C_q$ between any two given simple modules $L_q(\lambda)$ and $L_q(\mu)$, $\lambda, \mu \in X^+$, and likewise in $\mathcal C_p$. To do so we invoke the corresponding problem for $u_q$ and $G_1$, respectively. It turns out that one central difficulty when dealing with these problems is to verify that (with a single exception when $R$ has type $C$) there are no self-extensions of simple modules for $u_q$ and $G_1$. Once this fact is established we easily reduce the problem of computing the extensions between all pairs of simple modules in $\mathcal C_q$ and $\mathcal C_p$ to the problems of  determining the corresponding finitely many extensions for the simple modules for $u_q$ and $G_1$. In turn, our proof reveals that this problem for $u_q$ and $G_1$ is equivalent to finding the extensions in $\mathcal C_q$ and $\mathcal C_p$ between specified finite sets of simple modules.

Assume now that $\ell$ is larger than the Coxeter number for $R$, respectively that $p$ is so large that the Lusztig conjecture holds for $G$.  We conclude the paper by showing that the proofs of the Lusztig conjectures for the characters of simple modules allow us to prove that the dimensions of the extensions group in $\mathcal C_q$ equals those in $\mathcal C_p$. In fact, we prove that both are given by specific  coefficients of the associated Kazhdan-Lusztig polynomials. In other words, the algorithm determining the KL-polynomials for affine Weyl group for $R$ gives not only ways of computing the simple characters in $\mathcal C_q$ and $\mathcal C_p$, but it also gives explicit algorithms for solving the extension problems in these categories (under the mentioned assumptions on $\ell$ and $p$).

\subsection{Extensions of simple modules for $G_1$ and for $u_q$}
 An old theorem due to Curtis \cite{Cu} says that the family $\{L_p(\lambda)\mid \lambda \in X_1\}$ consists of $G$-modules which stay irreducible when restricted to $G_1$. Moreover, up to isomorphisms this set is a full set of simple $G_1$-modules. 

A completely similar result holds for $u_q$ (see \cite{Lu90}) when we replace $p$ by $\ell$:
If $\lambda \in X_1$ then $L_q(\lambda)$ is simple as a $u_q$-module, and all simple $u_q$-modules of type $1$ are obtained in this way.

As mentioned in Section 1.2 a non-trivial part of determining all extensions between pairs of simple modules for $u_q$, respectively for $G_1$, is to prove that self-extensions of such modules almost never occur. In fact, we prove that such self-extensions are impossible except  when $R$ is of type $C$ and $p$, respectively $\ell$, is 2. We shall call this case {\it the very special case}   (`special' because only one type among all types of root systems is allowed, namely type $C$,  and `very special' because at the same time only one value of $p$, respectively $\ell$ is allowed, namely $p =2$, respectively $\ell = 2$). In that particular case we shall show that for instance the trivial module does in fact extend itself non-trivially.

Our main theorem on self-extensions can then be stated

\begin{thm} Except in {\it the very special case} there are no self-extensions between simple $G_1$-modules, nor between simple $u_q$-modules. 
\end{thm}
This will be proved in Section 3 (see Theorem 3.1) where we also show that in {\it the very special case} both $G_1$ and $u_q$ do have simple modules with non-trivial self-extensions. 

The first part of the following corollary is an immediate consequence of the above theorem while the second part comes from the fact that $G_1$ is a normal subgroup scheme of $G$, cf. \cite{A82} Section 5. 

\begin{cor} Let $\lambda, \mu \in X^+$. Then 
\begin{enumerate} 
\item Suppose we are not in {\it the very special case}. If $\lambda^0 =\mu^0$ then $\Ext_{\mathcal C_p}^1(L_p(\lambda), L_p(\mu)) \simeq \Ext^1_{\mathcal C_p}(L_p(\lambda^1), L_p(\mu^1))$.
\item If $\lambda^0 \neq \mu^0$ then $\Ext^1_{\mathcal C_p}(L_p(\lambda), L_p(\mu)) \simeq \Hom_{\mathcal C_p}(L_p(\lambda^1), \Ext^1_{G_1}(L_p(\lambda^0), L_p(\mu^0))^{(-1)} \otimes L_p(\mu^1))$.
\end{enumerate}
\end{cor}

In the category $\mathcal C_q$ Theorem 1.1 has the following even stronger consequence:

\begin{cor} Let $\lambda, \mu \in X^+$. 
\begin{enumerate}
\item Suppose we are not in {\it the very special case}. If $\lambda^0 = \mu^0$ then $\Ext^1_{\mathcal C_q}(L_p(\lambda), L_p(\mu)) = 0$. 
\item If $\lambda^0 \neq \mu^0$ then $\dim_{\C}\Ext^1_{\mathcal C_q}(L_q(\lambda), L_q(\mu)) = [\Ext^1_{u_q}(L_q(\lambda^0), L_q(\mu^0))^{[-q]}: L_0(\mu^1))]$.
\end{enumerate}
\end{cor}
We deduce this corollary in Section 4. Here we also deduce the facts that the problem of determining the extensions between pairs of simple modules in both $\mathcal C_q$ and $\mathcal C_p$ reduces to a specified finite set of simple modules.

Finally, in Section 5 we show how a solution of the extension problem in $\mathcal C_q$ can be given in turn of the Kazhdan-Lusztig polynomials. This holds for $\ell$ bigger than the Coxeter number. The same is true in $\mathcal C_p$ but here we need $p$ to be very large, \cite{AJS}, \cite{W}.

\subsection{Background and outline}

In the modular case the above extension problem was dealt with in my $40^+$-year old paper \cite{A82}. As we shall demonstrate many ingredients from this proof work just as well in the quantum case (which did not exist in the early 1980's). Usually, we therefore carry out proofs only in the quantum case and only point out the differences (if any) needed in the modular case.  Whenever possible we take advantage of newer results in order to simplify the proofs. Main examples of newer results that we have employed are: the two useful spectral sequences with the same abutment from [AJ], the treatment of the modular case by Jantzen in Chapter  12 of [RAG], and the computations of some $B$, respectively $B_q$-cohomology from \cite{An07} and \cite{AR}.

The paper is organized as follows. In Section 2 we list several ``five-term exact sequences'' which play a crucial role in our proofs in the following sections. Then in Section 3 the key results  on self-extensions for $U_q$ and $G_1$ are proved. These theorems are used in Section 4 to deduce the results on how to restrict the extension problem for simple modules in the quantum case as well as in the modular case to a finite set.  Finally, Section 5 contains the above described connection between these extensions and some specific coefficients of certain KL-polynomials.

\section{Five-term exact sequences}

In this section we shall consider several instances in which compositions of left exact functors between modular categories have right derived functors which are abutments of corresponding spectral sequences. In particular, we shall need the associated five-term exact sequences. In the modular cases some of these appeared already in \cite{A82}. We repeat them here because they play a key role in our proofs in Section 3. Moreover, we want to illustrate both the similarities and the differences to the quantum case.

\subsection{Compositions of certain left exact functors}
Let $\mathcal A$, $\mathcal B$ and $\mathcal C$ be abelian categories and let $F: \mathcal A \rightarrow \mathcal C$ be the composite of two left exact functors $F_1: \mathcal A \rightarrow  \mathcal B$ and $F_2: \mathcal B \rightarrow \mathcal C$. Assume that both $\mathcal A$ and $\mathcal B$ have enough injectives and that the functor $F_1$ takes injective objects in $\mathcal A$ into objects in $\mathcal B$ which are acyclic for $F_2$. Then we have for each object $M$ in $\mathcal A$ a spectral sequence
$$ (R^{i}F_2 \circ R^jF_1)(M) \implies R^{i+j}F (M).$$
The corresponding five-term exact sequence is 
$$ 0 \to R^1F_2 (F_1 (M)) \to R^1F(M) \to F_2 (R^1F_1 (M)) \to R^2F_2 (F_1 (M)) \to R^2F(M).$$

In this paper we shall need several such spectral sequences and their associated five-term exact sequences. We will start by giving these sequences in the modular case (compare with Section 1 of \cite {A82} ) and then list the corresponding results in the quantum case. As we shall see this latter case is similar and at some points somewhat simpler.

\subsection{The modular case}

Consider a group $K$ with a normal subgroup $N$. Then the fixed point functor taking a $K$-module $M$ into $M^K$ has right derived functors which we denote $H^{i}(K, -)$. With the analogous notations for the derived functors of the fixed point functors for $N$ and $K/N$ we have as an example of the above the socalled Lyndon-Hochschild-Serre spectral sequence for a $K$-module $M$

$$ H^{i}(K/N,H^j(N,M)  \implies H^{i+j}(K,M).$$

\subsubsection{The fixed point functors for $G$ and $B$}
Consider the fixed point functors for our algebraic group $G$ and for its normal subgroup scheme $G_1$. As the category $\mathcal C_p$ does not have enough injectives we enlarge it to the category $\tilde {\mathcal C_p}$ consisting of all rational $G$-modules. Then in analogy with the above we have for each $M \in \tilde{\mathcal C_p}$ the spectral sequence

$$ H^{i}(G,H^j(G_1,M)^{(-1)})  \implies H^{i+j}(G,M).$$
The corresponding five-term exact sequence
\begin{equation} \label{m-1} 0 \to H^1(G,H^0(G_1,M)^{(-1)}) \to H^1(G,M) \to H^0(G, H^1(G_1,M)^{(-1)})
\end{equation}
$$ \hskip 7 cm \to H^2(G,H^0(G_1,M)^{(-1)}) \to H^2(G,M).$$

Now the fixed point functor for $G$ is equal to $\Hom_G (k, -)$ ($k$ denoting the trivial $1$-dimensional $G$-module). The $i$'th right derived functor of $\Hom_G (k, -)$ is $\Ext_G^{i}(k,-)$. If the module $M$ in (\ref{m-1}) has the form $M = L \otimes M'$ for some $L$ with $L = (L_0 \otimes L_1^{(1)})^*$ (here and elsewhere the index $^*$ on a module means the dual module)
 where $L_0, L_1 \in \mathcal C_p$ and $M' \in \tilde{\mathcal C_p}$. Then the above five-term exact sequence is
\begin{equation} \label{m-2} 0 \to \Ext^1_G(L_1,\Hom_{G_1}(L_0,M')^{(-1)}) \to \Ext^1_G(L,M) \to \Hom_G(L_1,\Ext^1_{G_1}(L_0,M')^{(-1)})
\end{equation} 
$$ \hskip 7cm \to \Ext^2_G(L_1, \Hom_{G_1}(L_0,M')^{(-1)}) \to \Ext_G^2(L,M).$$

The above sequences clearly have analogues when we replace $G$ by a closed subgroup. In particular, we shall need them in the case of the Borel subgroup $B$ of $G$. Here we have the following five-term sequence:

Let $L, E $ be two finite dimensional $B$-modules and let $\mu \in X$. Then in analogy with the above we have the exact sequence
\begin{equation} \label{m-3} 0 \to \Ext^1_B(\mu, \Hom_{B_1}(L, E)^{(-1)}) \to \Ext^1_B(L\otimes \ell \mu, E) \to \Hom_B(\mu, \Ext^1_{B_1}(L, E)^{(-1)})
\end{equation}
 $$ \hskip 7cm \to \Ext^2_B(\mu, \Hom_{B_1}(L,E)^{(-1)}) \to \Ext_B^2(L \otimes \ell \mu,E).$$
Here we could of course replace $\mu$ by an arbitrary finite dimensional $B$-module.

\subsubsection{Induction functors}

Let $\tilde {\mathcal B}$ denote the category of rational $B$-modules. The induction functor  $\Ind_B^G: \tilde {\mathcal B} \to \tilde {\mathcal C_p}$ will also be denoted $H^0(G/B,-)$ and its derived functors $H^{i}(G/B, -)$. In the following sections we often abbreviate and write  $H^{i}_p$ instead of $H^{i}(G/B, -)$, respectively $H^{i}_0$ instead of $H^{i}(G_0/B_0, -)$.  Note that the composite $H^0(G, H^0(G/B, - ))$ equals the $B$-fixed point functor $H^0(B, -)$. Therefore we have  for each rational $B$-module $E$ a spectral sequence
$$ H^{i}(G, H^j(G/B, E)) \implies H^{i+j}(B,E)$$
with corresponding five-term exact sequence
\begin{equation} \label{m-4}0 \to  H^1(G, H^0(G/B, E)) \to H^1(B, E) \to H^0(G, H^1(G/B, E)) 
\end{equation}
$$ \hskip 7 cm \to H^2(G, H^0(G/B, E)) \to H^2(B, E).$$
Now if $M$ is a finite dimensional $G$-module we have the tensor identity $H^j(G/B, M \otimes E) \simeq M \otimes H^j(G/B, E)$. As in the previous section this leads to the following five-term exact sequence (valid for any $M \in \mathcal C_p$ and any rational $B$-module $E$)
\begin{equation} \label{m-5}0 \to  \Ext_G^1(M, H^0(G/B,E)) \to \Ext_B^1(M,E) \to \Hom_G(M, H^1(G/B, E)) 
\end{equation}
$$ \hskip 5cm \to 
 \Ext_G^2(M, H^0(G/B,E)) \to \Ext_B^2(M,E).$$
 
\subsubsection{Relations between the $B_1$- and the $G_1$-fixed point functors}
It is not hard to establish an equivalence between the following two compositions of functors 
$$\Ind_B^G \circ H^0(B_1, -)^{(-1)} \text { and }H^0(G_1, \Ind_B^G -)^{(-1)}$$
(for a proof see \cite{AJ}, Proposition 3.1).

Denote this functor $F$. Then we have two spectral sequences with abutment $R^{i}F$ and hence for any rational $B$-module $E$ the following two five-term exact sequences
\begin{equation} \label{m-6}0 \to H^1(G/B, H^0(B_1,E)^{(-1)}) \to R^1F(E) \to H^0(G/B, H^1(B_1, E)^{(-1)})
\end{equation} 
$$\hskip 7cm \to 
H^2(G/B, H^0(B_1,E)^{(-1)}) \to R^2F(E),$$
and 
\begin{equation} \label{m-7}0 \to H^1(G_1, H^0(G/B, E))^{(-1)} \to R^1F(E) \to H^0(G_1, H^1(G/B, E))^{(-1)} 
\end{equation} 
$$\to
H^2(G_1, H^0(G/B,E ))^{(-1)} \to R^2F(E).$$
 
\subsection{The quantum case}

In the quantum case the Frobenius homomorphism is not an endo-functor on $U_q$. Instead it is a functor $Fr_q: U_q \to U_0$. As mentioned in Section 1 the category $\mathcal C_0$ is semisimple, and this fact makes some of the $q$-analogues of the results in Sections 2.2 easier. Moreover, $\mathcal C_q$ has enough injectives (in fact the linkage principle in $\mathcal C_q$, see \cite{An03}, implies that the $q$-Steinberg module $St_q = L_q((\ell - 1)\rho)$ is injective and the natural inclusion of the trivial module $L_q(0) = \C$ into $St_q \otimes St_q$ gives that any $M \in \mathcal C_q$ is a submodule of the injective module $St_q \otimes St_q \otimes M$). This means that we do not need to enlarge $\mathcal C_q$ like we did in the case $\mathcal C_p$.

Below we list the $q$-analogues of the key results from Section 2.2.

\subsubsection{Quantum fixed point functors}

As $\mathcal C_0$ is semisimple the spectral sequence corresponding to the modular analogue in Section 2.2.1 reduces to the following isomorphism valid for all $M \in \mathcal C_q$ and all $j \geq 0$
$$ H^0(U_0, H^j(u_q, M)^{[-q]}) \simeq H^j(U_q, M).$$
Replacing $M$ by a module of the form $L^* \otimes M$ with $L = L_0 \otimes L_1^{[q]}$,  $L_0, M \in \mathcal C_q$ and $L_1 \in \mathcal C_0$ this may also be formulated as follows 
\begin{equation} \label{q-1} \Ext^j_{\mathcal C_q} (L,M) \simeq \Hom_{\mathcal C_0}(L_1, \Ext^j_{u_q}(L_0, M)^{[-q]}).
\end{equation}

Consider the subalgebra $B_q = U_q^- U_q^0$ of $U_q$ and let $b_q$ denote the small quantum group for $B_q$ then the five-term sequence analogous to  (\ref{m-3}) is
\begin{equation} \label{q-2} 0 \to \Ext^1_{B_0}(L_1, \Hom_{b_q}(L_0, M)^{[-q]}) \to \Ext^1_{B_q}(L,M) \to \Hom_{B_0}(L_1, \Ext^1_{b_q}(L_0,M)^{[-q]})
\end{equation}
$$ \hskip 7cm \to \Ext^2_{B_0}(L_1, \Hom_{b_q}(L_0, M)^{[-q]}) \to \Ext^2_{B_q}(L,M).$$
Again $L = L_0 \otimes L_1^{[q]}$ and we assume that both  $L_0$ and $M$ are finite dimensional $B_q$-modules, while $L_1$ is a finite dimensional $B_0$-module.

\subsubsection{Quantum induction functors}
Denote by $H^0(U_q/B_q, - )$ the quantum induction functor from $B_q$ to $U_q$, see \cite{APW}, and let $H^j(U_q/B_q, - )$ be its right derived functors (in the following sections denoted just $H^j_q$).  Then in analogy with (\ref{m-5})
we have for each $M \in \mathcal C_q$ and each finite dimensional $B_q$-module $E$ the five-term exact sequence
\begin{equation} \label{q-3} 0 \to \Ext_{\mathcal C_q}^1(M, H^0(U_q/B_q,E)) \to \Ext_{B_q}^1(M,E) \to \Hom_{\mathcal C_q}(M, H^1(U_q/B_q, E)) 
\end{equation}
$$ \hskip 7 cm \to \Ext_{\mathcal C_q}^2(M, H^0(U_q/B_q,E)) \to \Ext_{B_q}^2(M,E).$$

\subsubsection{Relations between $b_q$- and $u_q$- fixed point functors} 
Define a functor $F_q$ from the category of finite dimensional $B_q$-modules to $\mathcal C_0$ by
$$ H^0(G_0/B_0, H^0(b_q, -)^{[-q]}) = F_q = H^0(u_q, H^0(U_q/B_q, -))^{[-q]}.$$
Then in analogy with (\ref{m-6}) and (\ref{m-7}) we have for any finite dimensional $B_q$-module $E$ 
the following two five-term exact sequences
\begin{equation} \label{q-4} 0 \to H^1(G_0/B_0, H^0(b_q,E)^{[-q]}) \to R^1F_q(E) \to H^0(G_0/B_0, H^1(b_q, E)^{[-q]})
\end{equation}
$$\hskip 7cm \to H^2(G_0/B_0, H^0(b_q,E)^{[-q]}) \to R^2F_q(E),$$
and 
\begin{equation} \label{q-5} 0 \to H^1(u_q, H^0(U_q/B_q, E))^{[-q]} \to R^1F_q(E) \to H^0(u_q, H^1(U_q/B_q, E))^{[-q]}
\end{equation} 
$$ \hskip 7 cm \to H^2(u_q, H^0(U_q/B_q,E ))^{[-q])} \to R^2F_q(E).$$

\section{Self-extensions of simple modules for $G_1$ and for $u_q$}
It is wellknown/easy to prove that for any $\lambda \in X^+$ we have
$$ \Ext^1_{\mathcal C_p} (L_p(\lambda, L_p(\lambda)) = 0 = \Ext_{\mathcal C_q}^1(L_q(\lambda), L_q(\lambda)),$$
i.e. in $\mathcal C_p$ and in $\mathcal C_q$ there are no self-extensions of simple modules. 

In this section we prove that the same is almost always the case when we pass to finite dimensional $G_1$-modules, respectively $u_q$-modules.

\subsection{Self-extensions for $G_1$- and $u_q$-modules}

Recall that by {\it the very special case} we understand the case where $R$ is of type $C$ and $\ell = 2$, respectively $p =2$. Whenever we are in this case we will in this and the following sections use the Bourbaki notation from Planche III in \cite{B} (except that we will use $n$ instead of $\ell$ to denote the rank of $R$). In particular, the unique simple root in $C_n$ belonging to $2X$ is $\alpha_n = 2 \epsilon_n$. 

\begin{thm} Assume we are not in {\it the very special case}. Then we have for all $\lambda \in X_1$ 
\begin{enumerate}
\item $\Ext^1_{u_q}(L_q(\lambda), L_q(\lambda)) = 0$,
\item $\Ext^1_{G_1}(L_p(\lambda), L_p(\lambda)) = 0$.
\end{enumerate}
\end{thm}
Below we give a proof of this theorem for the quantum case. We  follow to some extent the strategy for the modular case, see \cite{A82}.

In Section 3.4 we give examples which show that in {\it the very special case} self-extensions do occur.

\subsection{A key lemma}
The following lemma will be the key ingredient in our proof of Theorem 3.1, see Section 3.3 below.

\begin{lem} Let $\lambda \in X_1$. 
\begin{enumerate}
\item $\Hom_{b_q}(L_q(\lambda), \mu) = \delta_{\lambda, \mu} \C$ for all $\mu \in X_1$.
\item Assume we are not in {\it the very special case}. Then $\Ext^1_{b_q}(L_q(\lambda), \lambda) = 0.$
\item In {\it the very special case} we have for all $\mu \in X$
$$ \Hom_{B_0}(\mu, \Ext_{b_q}^1(L_q(\lambda), \lambda)^{[-q]}) = \begin{cases} {\C \text { if } \mu = \epsilon_n \text { and } \langle \lambda, \alpha_n^\vee \rangle = 0,}\\
{0} \text{ otherwise.} \end{cases}$$
 
\end{enumerate}
\end{lem}
\begin{proof} Let $\mu \in X_1$ and set $Z_q(\mu) = \Ind_{b_q}^{u_q} (\mu)$. Then $Z_q(\mu)$ has simple $u_q$-socle equal to $L_q(\mu)$. Hence (1) follows via the identity $\Hom_{b_q}(L_q(\lambda), \mu) =  \Hom_{u_q}(L_q(\lambda), Z_q(\mu))$.

To prove (2) and (3) we set $E_j = \Ext^j_{b_q}(L_q(\lambda), \lambda)^{[-q]}$ and observe that these  $B_0$-modules are finite dimensional. According to (1) we have $E_0 = \C$. Clearly,  $E_1 = 0$ if and only if $\Hom_{B_0}(\nu, E_1) = 0$ for all $\nu \in X$. So consider $\nu \in X$. By (\ref{q-2}) we have the exact sequence
\begin{equation} \label{key} 0 \to \Ext^1_{B_0}(\nu, \C) \to \Ext^1_{B_q}(L_q(\lambda), \lambda - \ell \nu) \to \Hom_{B_0} (\nu, E_1) \to  \Ext^2_{B_0}(\nu, \C).
\end{equation}
We now have the following claims about the terms in this sequence
\vskip 0,5 cm
{\bf Claim 1} The first term is $\C$ if $\nu \in S$, and $0$ otherwise.

{\bf Claim 2} Assume we are not in {\it the very special case}. Then the second term is $\C$ if $\nu \in S$, and $0$ otherwise.

{\bf Claim 3} In {\it the very special case} the second term is $\C$ if either $\nu \in S$,  or $\nu = \epsilon_n$ and $\langle \lambda, \alpha_n \rangle = 0$, 
and $0$ otherwise.  

{\bf Claim 4} The fourth term is $\C$ if $\nu = - s_\alpha s_\beta \cdot 0$ for some $\alpha, \beta \in S$ with $\alpha \neq \beta$, and $0$ otherwise.

{\bf Claim 5} If $\mu$ is a weight of $\Ext^1_{b_q}(L_q(\lambda), \lambda)^{[-q]}$ then $\ell \mu = \sum_{\gamma \in R^+} n_\gamma \gamma$ for some $(n_\gamma)_{\gamma \in R^+}$ with $0 \leq n_\gamma \leq \ell -1$ for all $\gamma$. 
\vskip 0,5 cm
Here {\bf  Claim 1} and {\bf Claim 4} follow from Bott's theorem (\cite{Bo}) by noticing that we have 
$$\Ext^j_{B_0}(\nu, \C) \simeq H^j(B_0, -\nu) \simeq H^0(G_0, H^j_0(-\nu)).$$ 
The last term here is $\C$ if there exists $w \in W$ of length $j$ such that $w \cdot (-\nu) = 0$, and $0$ otherwise.

Consider then {\bf Claim 2} and {\bf Claim 3}. We first observe that the second term in (\ref{key}) equals $\Ext^1_{\mathcal C_q}(L_q(\lambda), H^0_q(\lambda -\ell \nu))$ if $\lambda - \ell \nu \in X^+$. 
Now the Weyl module $\Delta_q(\lambda) \in \mathcal C_q$ surjects onto $L_q(\lambda)$ and the kernel $K$ of this surjection has composition factors with highest weights $\mu$ that are all strictly less than $\lambda$. The short exact sequence $0 \to K \to \Delta_q(\lambda) \to L_q(\lambda) \to 0$ and the fact that $\Ext_{\mathcal C_q}^1(\Delta_q(\lambda), \nabla_q(\mu)) = 0$ for all $\mu \in X^+$ (here $\nabla_q(\mu) =  H_q^0(\mu)$ is the dual Weyl module in $\mathcal C_q$ with highest weight $\mu$)  imply that $\Ext^1_{\mathcal C_q}(L_q(\lambda), \nabla_q(\lambda -\ell \nu))$ is a quotient of $\Hom_{\mathcal C_q} (K, \nabla_q(\lambda - \ell \nu))$. Hence if $\Ext^1_{\mathcal C_q}(L_q(\lambda), \nabla_q(\lambda -\ell \nu))$ is non-zero, then we must have that  $L_q(\lambda - \ell \nu)$ is a composition factor of $K$. If that is the case then $\lambda - \ell \nu < \lambda$, i.e. $\ell \nu > 0$. However, our assumption $\lambda - \ell \nu \in X^+$ means that $\langle \nu, \alpha^\vee \rangle \leq 0$ for all $\alpha \in S$, and this makes the inequality  $\ell \nu > 0$  impossible. 

On the other hand, if $\lambda - \ell  \nu \notin X^+$ then by the quantum Kempf's vanishing theorem  (\cite{SRH}) we get $\Ext^1_{B_q}(L_q(\lambda), \lambda - \ell \nu) \simeq \Hom_{\mathcal C_q}(L_q(\lambda), H^1_q(\lambda - \ell \nu))$.
Recall that if $H^1_q(\lambda - \ell \nu)) \neq 0$ then it has simple socle (see \cite{A81}) with highest weight $\lambda - \ell \nu + a \alpha$ for an appropriate simple root $\alpha$ and some 
$a \in \Z_{>0}$. We see that this highest weight only equals $\lambda$ if $\ell \nu = a \alpha$.  When we are not in {\it the very special case } it follows that   $L_q(\lambda) \subset H^1_q(\lambda - \ell \nu)$ if and only if $\nu \in S$. This proves {\bf Claim 2}. 

In {\it the very special case} similar arguments as above show that $H^1_q(\lambda - 2 \nu)$ has socle $L_q(\lambda)$ if either $\nu \in S$  or $\nu = \epsilon_n$  This proves {\bf Claim 3}.  

Finally, to prove {\bf Claim 5} we
consider the inclusion of $B_q$-modules $\lambda \subset St_q \otimes ((\ell - 1)\rho + \lambda)$ and denote the quotient  by $Q$. As $St_q$ is an injective $U_q$-module it is also injective for $u_q$ and $b_q$. Hence the short exact sequence of $B_q$-modules
$$ 0 \to \lambda \to St_q \otimes ((\ell - 1)\rho + \lambda) \to Q \to 0$$
gives a surjection $\Hom_{b_q}(L_q(\lambda), Q) \to \Ext_{b_q}^{1}(L_q(\lambda), \lambda)$. Using part (1) of the Lemma this implies that the weights $\mu$ of $E_1$ must satisfy $\ell \mu = \sum_{\gamma \in R^+} n_\gamma \gamma$ for some $(n_\gamma)_{\gamma \in R^+}$ with $0 \leq n_\gamma \leq \ell -1$ as claimed.  
\vskip ,5 cm 
Now we turn to the proof of (2) in our lemma. Here we are not in {\it the very special case} and the combination of  Claims (1) and (2) shows that the first two terms in (\ref{key}) coincide. Then Claim 4 implies that (2) holds except possibly when $\nu = - s_\alpha s_\beta \cdot 0$  for some simple roots $\alpha, \beta$ with $\alpha \neq \beta$.  But a weight  $\nu$ of this form cannot satisfy Claim (5). In fact, note that  $- s_\alpha s_\beta \cdot 0 = a\alpha + \beta$ with $a = -\langle \beta, \alpha^\vee \rangle  + 1 \in \{1, 2, 3, 4\}$. 
So the equality $- \ell  s_\alpha s_\beta \cdot 0 = \sum_\gamma n_\gamma \gamma$ says
$\ell(a\alpha + \beta) = \sum_\gamma n_\gamma \gamma$. In particular, $n_\gamma = 0$ unless $\gamma \in \spa_{\Z_{\geq 0}}\{\alpha, \beta\}$.
If $a=1$ we have $\spa_{\Z_{\geq 0}}\{\alpha, \beta\} \cap R^+ =\{\alpha, \beta\}$ so the equality is $\ell(\alpha + \beta) = n_\alpha \alpha + n_\beta \beta$.  This is clearly incompatible with the inequalities on $n_\alpha$ and $n_\beta$. We leave to the reader to check the remaining $3$ possible values of $a$. 

Finally, we prove part (3). Here we first observe that the last part of the proof of (2) also works in {\it the very special case}, i.e. the map between the third and fourth terms of (\ref{key}) is always zero. By Claims (1) and (3) we then deduce that 
$$ \Hom_{B_0}(\nu, E_1) = \C \text { if } \nu = \epsilon_n \text { and } \langle \lambda, \alpha_n^\vee \rangle = 0,$$
and that otherwise $ \Hom_{B_0}(\nu, E_1) = 0$.

\end{proof}

\subsection{Proof of Theorem 3.1(1)}

Consider the short exact sequence
$$ 0 \to L_q(\lambda) \to H^0_q(\lambda) \to Q = H^0_q(\lambda)/L_q(\lambda) \to 0.$$
Note that $\Hom_{u_q}(L_q(\lambda), Q) = 0$ as otherwise we would have that for some $\nu \in X^+$ the simple $U_q$-module $L_q(\lambda) \otimes L_0(\nu)^{[q]}$ would be a submodule of $Q$. But $\lambda + \ell \nu$ is not a weight of $Q$ for any $\nu \in X^+$. 

To finish the proof it will therefore be enough to check that 
\begin{equation} \label{ext}
\Ext^1_{u_q}(L_q(\lambda), H^0_q(\lambda)) = 0.
\end{equation}

To see this we will use the two exact sequences (\ref{q-4}) and (\ref{q-5}) with $E = L_q(\lambda)^* \otimes \lambda$. For this module the first term in (\ref{q-4}) equals $H^1_0( \Hom_{b_q}(L_q(\lambda), \lambda)^{[-q]})$. But by Lemma 3.2(1) this is isomorphic to $H^1_0(\C)$, which is zero by Bott's theorem (\cite{Bo}).

The third term in (\ref{q-4}) is $H^0_0 (\Ext^1_{b_q}(L_q(\lambda), \lambda)^{[-q]})$  and this vanishes by Lemma 3.2(2). 

We conclude thus that the second term $R^1F_q(L_q(\lambda)^*\otimes \lambda) = 0$. Inserting this in (\ref{q-5}) we get that the first term of this sequence is $0$. This is exactly the vanishing we are looking for.

\subsection{The very special case}

 In this subsection we consider the case where $U_q$ is the quantum group associated with  $\mathfrak g= \mathfrak{sp}_{2n}$ of type $C$ and $\ell = 2$
 (i.e. $q \in \{\pm i\}$).  
 
Recall that the simple root in $R \cap 2X$ is $\alpha_n = 2 \epsilon_n$ and that by Lemma 3.2(3) we have
$$ \Ext^1_{b_q}(L_q(\lambda), \lambda) = 0 \text { for all } \lambda \in X_1 \text { with } \langle \lambda, \alpha_n^\vee \rangle = 1.$$

The very same arguments as used above in the proof of Theorem 3.1.(1) then gives
\begin{equation}
\Ext^1_{u_q}(L_q(\lambda, L_q(\lambda)) = 0 \text { for all } \lambda \in X_1 \text { with } \langle \lambda, \alpha_n^\vee \rangle = 1.
\end{equation}

We shall now show that when $\langle \lambda, \alpha_n^\vee \rangle = 0$ self-extensions of $L_q(\lambda)$ can exist.
\vskip ,5 cm
 
 Consider the natural module $V$ for the complex algebraic group $G_0 = Sp_{2n}$. This module is $2n$-dimensional with weights $\{\pm \epsilon_i\mid i = 1, 2, \cdots , n\}$, each occuring with multiplicity $1$. 
 We set $V^+ = V_{\epsilon_1} \oplus V_{\epsilon_2} \oplus \cdots \oplus V_{\epsilon_n}$  and $V^- = V_{-\epsilon_1} \oplus V_{-\epsilon_2} \oplus \cdots \oplus V_{-\epsilon_n}$. Then $V^-$ is a $B_0$-subspace of $V$, whereas $V^+$ is a $B_0$-quotient of $V$. So we have the following short exact sequence of $B_0$-modules
$$ 0 \to V^- \to V \to V^+ \to 0.$$
Note that $V \simeq V^*$ (as $G_0$-modules) so that $V^+ \simeq (V^-)^*$ (as $B_0$-modules). 

We may identify $V$ with the  Weyl module $\Delta_0(\epsilon_1)$ for $G_0$. Note that the $B_0$-module $V \otimes \epsilon_1$ has exactly $3$ dominant weights, namely $2 \epsilon_1 = 2 \omega_1,  \epsilon_1 + \epsilon_2 = \omega_2$, and $0$. All remaining weights $\nu$ satisfy $\langle \nu, \alpha^\vee \rangle = -1$ for some $\alpha \in S$. It follows that
$$ V \otimes V = \Delta_0(2\omega_1) \oplus \Delta_0(\omega_2) \oplus \Delta_0(0).$$
We have (e.g. by Weyl's dimension formula)
$$ \dim \Delta_0(2 \omega_1) = 2 n^2 + n,$$ 
$$ \dim \Delta_0(\omega_2) = 2 n^2 - n + 1,$$
and of course 
$$ \dim \Delta_0(0) = 1.$$

Then in $\mathcal C_q$ we have 
\begin{lem}
The Weyl module $\Delta_q(2\omega_1)$ has composition factors $L_q(2\omega_1) = V^{[q]}$, $L_q(\omega_2) = \Delta_q(\omega_2)$ and $L_q(0) = \C$ (all with multiplicity $1$). Hence in particular 
$\Ext^1_{\mathcal C_q} (L_q(2\omega_1), L_q(0)) = \C$.
\end{lem}

\begin{proof}
Let $\lambda \in X^+$. Then the quantum sum formula, \cite{APW}, (in the case where $\ell = 2$) says that the ``Jantzen filtration'' for the Weyl module $\Delta_q(\lambda)$ in $\mathcal C_q$
$$0 = \Delta_q(\lambda)^{r+1} \subset   \Delta_q(\lambda)^{r} \subset \cdots  \Delta_q(\lambda)^{1} \subset  \Delta_q(\lambda)$$
satisfies
$$ \sum_{j \geq 1}\cha(\Delta_q(\lambda)^j) = \sum_{\beta} \sum_m (-1)^{m-1} \chi(\lambda - m \beta),$$
where the first  sum is taken over those $\beta \in R^+$ for which $\langle \lambda + \rho, \beta^\vee \rangle$ is odd and the second sum over those $m$ for which $0 < 2m < \langle \lambda + \rho, \beta^\vee \rangle$.
Recall also that we always have $\Delta_q(\lambda)/\Delta_q(\lambda)^1 \simeq L_q(\lambda)$. 

These results imply easily
\begin{equation} \Delta_q(\omega_2) = L_q(\omega_2)
\end{equation} \label{a}
(because $\Delta_q(\omega_2)^1 = 0$), and
\begin{equation} \label{b}\Delta_q(2 \omega_1)^1 = L_q(\omega_2) \oplus L_q(0).\end{equation} 
Here the two simple modules on the right hand side cannot extend eachother because both are Weyl modules as well as dual Weyl modules (cf. (\ref{a})). 

Note that (\ref{b})  implies that 
\begin{equation} \label{c} \Ext^1_{\mathcal C_q}(V^{[q]}, \C) = \C = \Ext^1_{\mathcal C_q}(V^{[q]}, L_q(\omega_2)).\end{equation}
Here the first equality is identical to the last claim in the lemma.
\end{proof}

\begin{prop} The natural module $V$ for $G_0$ is contained in $H^1(u_q, \C)^{[-q]}$, i.e the trivial module for $u_q$ has at least $\dim V = 2n$ different self-extensions. 
\end{prop}

\begin{proof}
By (\ref{q-1}) with $L_0 = M = \C$ and $L_1 = V$ we have $\Ext^1_{\mathcal C_q}(V^{[q]}, \C) \simeq \Hom_{\mathcal C_0}(V, H^1(u_q, \C)^{[-q]})$. By (\ref{c}) the left hand side equals $\C$, and since $V$ is simple in $\mathcal C_0$ this means that $V \subset H^1(u_q, \C)^{[-q]}$.
\end{proof}

\begin{rem}
\begin{enumerate}
\item Using (\ref{q-4}) and (\ref{q-5}) we get that 
$$ H^1(u_q, \C)^{[-q]} \simeq H^0_{0}(H^1(b_q, \C)^{[-q]}).$$
(in fact this equality holds in all cases but when we are not in the very special case Lemma 3.2(2) reveals that it just reads $0 = 0$). 
Via  Proposition 3.4 this gives 
\begin{equation} \label{d} \C = \Hom_{\mathcal C_0}(V, H^1(u_q, \C)^{[-q]}) = \Hom_{B_0}(V, H^1(b_q, \C)^{[-q]})
\end{equation}
When we then combine  this with Lemma 3.2(3) we see that $V^+ $ is a $B_0$-submodule of $H^1(b_q, \C)^{[-q]}$, i.e there are at least $n$ different $b_q$ self-extensions of the trivial module.
\item
Note that $L_q(\omega_1) = \Delta_q(\omega_1)$. Comparing with the above calculation of $V \otimes V$ we see that the tensor product $L_q(\omega_1) \otimes L_q(\omega_1)$ has a  $\Delta_q$-filtration with factors $\Delta_q(2 \omega_1), \Delta_q(\omega_2)$ and $\Delta_q(0) = \C$. 

Set $Q = L_q(\omega_1) \otimes L_q(\omega_1)/\C$. The above results on the Jantzen filtations of $\Delta_q(2\omega_1)$ and $\Delta_q(\omega_2)$ implies that the $U_q$-socle of $Q$ equals $L_q(\omega_2)$. It follows that $H^0(u_q, Q) = 0$. Hence the short exact sequence $0 \to \C \to L_q(\omega_1) \otimes L_q(\omega_1) \to Q \to 0$ implies that $H_1(u_q, \C)$ is contained in $\Ext^1_{u_q}(L_q(\omega_1), L_q(\omega_1))$. Hence Proposition 3.4 implies that the simple $u_q$-module $L_q(\omega_1)$ also has non-trivial self-extensions.
\end{enumerate}
\end{rem}

\begin{rem}
\item In the modular case the analogue of Proposition 3.4 also holds. Note however, that Lemma 3.3 does not carry over verbatim to the modular case (In type $C$ the modular Weyl module $\Delta_q(\omega_2)$ is not simple and the Weyl module $\Delta_2(2 \omega_1)$ has $4$ instead of $3$ composition factors. This comes from the more complicated sum formulas in the modular case, see \cite{An83}. 
Nevertheless one may check that in the modular case we still have $V^{(1)} \subset H^1(G_1, k)$ and $(V^+)^{(1)} \subset H^1(B_1, k)$. In fact, equality holds (in both cases), see the calculations in \cite{AJ}, Section 6.19.

\end{rem}

\section{Extensions between pairs of simple modules in $\mathcal C_q$ and $\mathcal C_p$}

In this section we shall show how to obtain all extensions between any two simple modules in $\mathcal C_q$, respectively $\mathcal C_p$, from the extensions of the finitely many pairs of simple modules for $u_q$, respectively $G_1$. Alternatively, we can turn things around and obtain all extensions for simple $u_q-$, respectively $G_1-$, modules from a (finite) set of simple modules in $\mathcal C_q$, respectively $\mathcal C_p$, see Remark 4.3(2) below.

We shall need the following notation. Let $\alpha_0$ denote the highest short root in $R$. Moreover, we denote the Coxeter number of $R$ by $h$. We have $h = \langle \rho, \alpha_0^\vee \rangle -1$.  Then we set 
$$ A_0 = \{\nu \in X^+ \mid \langle \nu, \alpha_0^\vee \rangle < 2(h-1)\}.$$
Note that $A_0$ is a finite subset of $X^+$.

\subsection{The quantum case}

We begin with a lemma which gives a bound on the weights of all the ext-groups $\Ext^1_{u_q}(L_q(\lambda), L_q(\mu))$ with $\lambda, \mu \in X_1$ .

\begin{lem} Let $\lambda, \mu \in X_1$. Then any composition factor $L_0(\nu)$ of $\Ext^1_{u_q}(L_q(\lambda), L_q(\mu))^{[-q]}$ have $\nu \in A_0$. 
\end{lem}

\begin{proof}
When $\mu \in X_1$ we have $L_q((\ell - 1)\rho + w_0\mu)^* = L_q((\ell -1) \rho - \mu)$. It follows that $L_q(\mu) \otimes L_q((\ell - 1)\rho + w_0\mu)^*$ has a unique $U_q$-homomorphism onto $St_q$ and that we have a corresponding injection $L_q(\mu) \subset St_q \otimes L_q((\ell -1)\rho + w_0 \mu)$. Let $Q$ denote the cokernel of this map. Then the short exact sequence
$$ 0 \to L_q(\mu) \to St_q \otimes L_q((\ell -1)\rho + w_0 \mu) \to Q \to 0$$
combined with the fact that $St_q$ in injective in $\mathcal C_q$ (and hence is also injective as a $u_q$-module) show that $\Ext_{u_q}^1(L_q(\lambda), L_q(\mu))$ is a quotient of $\Hom_{u_q}(L_q(\lambda), Q)$. 

Suppose $L_0(\nu)$ is a summand of $\Ext_{u_q}^1(L_q(\lambda), L_q(\mu))^{[-q]}$. It then follows that $L_0(\nu)$ is also a summand of $\Hom_{u_q}^1(L_q(\lambda), Q)^{[-1]}$. 
However, $L_q(\lambda) \otimes \Hom_{u_q}^1(L_q(\lambda), Q)$ is a $U_q$-submodule of $Q$. We conclude that $\lambda + \ell \nu$ is a weight of $Q$. Hence $\lambda + \ell \nu < 2(\ell -1)\rho + \lambda$. 
This gives the inequality $\langle \lambda + \ell \nu, \alpha_0^\vee \rangle \leq \langle 2(\ell -1)\rho + \lambda, \alpha^\vee \rangle$ and hence ensures that $\nu \in A_0$. 
\end{proof}

Let now $\lambda, \mu \in X^+$ and write as usual $\lambda = \lambda^0 + \ell \lambda^1$, $\mu = \mu^0 + \ell \mu^1$ with $\lambda^0, \mu^0 \in X_1$ and $\lambda^1, \mu^1 \in X^+$. Then we have

\begin{thm} 
\begin{enumerate}
\item Assume that we are not in {\it the very special case}. If $\lambda^0 = \mu^0$ then $\Ext_{\mathcal C_q}^1(L_q(\lambda), L_q(\mu)) = 0.$ 
\item If $\lambda^0 \neq \mu^0$ we set $E^1(\lambda^0, \mu^0) = \Ext^1_{u_q}(L_q(\lambda^0), L_q(\mu^0))^{[-q]}$. Then \\
 $\dim_{\C} \Ext_{\mathcal C_q}^1(L_q(\lambda), L_q(\mu))  =  \sum_{\nu \in A_0} [E^1(\lambda^0, \mu^0):L_0(\nu)] [L_0(\nu) \otimes L_0(\mu^1): L_0(\lambda^1)].$
\end{enumerate}
\end{thm}

\begin{proof}
According to (2.8) we have in general 
$$\Ext^1_{\mathcal C_q}(L_q(\lambda), L_q(\mu)) = \Hom_{\mathcal C_0}(L_0(\lambda^1), \Ext^1_{u_q}(L_q(\lambda^0), L_q(\mu^0)^{[-q]} \otimes L_0(\mu^1)).$$ 
Then (1) follows from Theorem 3.1(1) and (2) from Lemma 4.1.

\end{proof}

\begin{rem}\begin{enumerate}
\item Note that  the right hand side of the equation in Theorem 4.1(2) is in fact ``computable''  once we know the characters of all $\Ext^1$-groups between simple modules of $u_q$, i.e. once we have found the finitely many numbers  $[\Ext^1_{u_q}(L_q(\lambda^0), L_q(\mu^0))^{[-q]}:L_0(\nu)]$ ($\lambda^0$ and $ \mu^0$ run through $X_1$ and $\nu \in A_0$). Of course to get the extensions between all simple modules in $\mathcal C_q$ we also have to know the infinitely many multiplicities $ [L_0(\nu) \otimes L_0(\mu^1): L_0(\lambda^1)]$. Here only $\nu$ is limited to a finite set, whereas $\lambda^1$ and $\mu^1$ run through the infinite set $X^+$. However, these multiplicities are well known classical quantities (they are e.g. Littlewood-Richardson coefficients). 

\item Conversely, Theorem 4.1(2) also tells us how to compute extensions of simple $u_q$-modules in terms of extensions of simple modules in $\mathcal C_q$. In fact, let $\lambda^0, \mu^0 \in X_1$ be given and set $\mu^1 = 0 $. Then the formula in Theorem 4.1 (2) says that 
\begin{equation} \label{ext} [\Ext_{u_q}^1(L_q(\lambda^0), L_q(\mu^0))^{[-q]} : L_0(\lambda^1)] = \dim_{\C} \Ext^1_{\mathcal C_q}(L_q(\lambda), L_q(\mu^0))
\end{equation}
for all $\lambda^1 \in X^+$. 

Note that by Lemma 4.1 the left hand side in (\ref{ext}) is zero unless $\lambda^1 \in A_0$. So this formula determines all extensions between simple $u_q$-modules in terms of the extensions of the finitely many pairs of simple modules $(L_q(\lambda), L_q(\mu))$ with $\lambda^0, \mu \in X_1$ and $\lambda^1 \in A_0$.
\end{enumerate}
\end{rem}

\subsection{The modular case}
The analogue of Theorem 4.2 in $\mathcal C_p$ is

\begin{thm} 
\begin{enumerate}
\item Assume that either $R$ is not of type $C$ or that $p >2$. If $\lambda^0 = \mu^0$ then $\Ext_{\mathcal C_p}^1(L_p(\lambda), L_p(\mu)) = \Ext_{\mathcal C_p}^1(L_p(\lambda^1), L_p(\mu^1)) $
\item If $\lambda^0 \neq \mu^0$ then \\
 $\Ext_{\mathcal C_p}^1(L_p(\lambda), L_p(\mu))  =  \Hom_{\mathcal C_p}(L_p(\lambda^1), \Ext^1_{G_1}(L_p(\lambda^0), L_p(\mu^0))^{(-1)} \otimes L_p(\mu^1)).$
 
\end{enumerate}
\end{thm}
 
 This theorem is proved using analogous arguments to those in the proof of Theorem 4.2, cf. Section 5 in \cite{A82}. 
 
 \begin{rem} \begin{enumerate}
 \item
We have to iterate this theorem to find all extensions of simple modules in $\mathcal C_p$. In fact, suppose $\lambda^0 = \mu^0$ (i.e. we are in case (1)). Then we write $\lambda^1 = \lambda^{10}+ p \lambda^{11} $ and similarly for $\mu^1$. If also $\lambda^{10} = \mu^{10}$ then we use (1) once more. We repeat this until $\lambda^{1\cdots 1} \neq \mu^{1\cdots 1}$ (if this never happens we have $\lambda = \mu$ and in this case we know $\Ext^1_{\mathcal C_p}(L_p(\lambda), L_p(\mu)) = 0$). Then we apply (2). 
\item
Like in the quantum case we have also in the modular case that if $\lambda, \mu \in X_1$ then all dominant weights of $\Ext^1_{G_1}(L_p(\lambda), L_p(\mu))^{(-1)}$ belong to $A_0$. The proof is similar to the proof of Lemma 4.1 except that $St_p$ is not injective in $\mathcal C_p$. However, $St_p$ is an injective $G_1$-module and this suffices to carry over the proof.
\item 
Suppose $p \geq 3(h-1)$. Then $A_0$ is contained in the bottom alcove $A^+$ of $X^+$ (note that $A^+= \{\lambda \in X^+ \mid \langle \lambda + \rho, \alpha_0^\vee \rangle < p\}$). In that case all $\Ext^1_{G_1}(L_p(\lambda^0), L_p(\mu^0))^{(-1)}$ are direct sums of simple Weyl modules with highest weights in $A^+$.  This shows that we can proceed just as in Remark 4.3 to get the all extensions between simple modules in $\mathcal C_p$ from the corresponding extensions for $G_1$ and vice versa. See also the results in section 5 of \cite{A82}.
\end{enumerate}
\end{rem}

 \section{Extensions of simple modules in $\mathcal C_q$ and Kazhdan-Lusztig polynomials}
 In this section we shall prove that the extensions of simple modules in $\mathcal C_q$ are given by certain specific coefficients of the Kazhdan-Lusztig polynomials, in short KL polynomials, for the affine Weyl group associated with our root system $R$. In the modular case 
 the corresponding result (for finitely many such extensions) was obtained in Section 2 of \cite{A86}. Note that this result depends on the modular Lusztig conjecture (cf. \cite{Lu80}), which is now known to hold only if $p \gg 0$, see \cite{AJS} and\cite{W}.
 
 In the quantum case the corresponding Lusztig conjecture (announced 8 years after the modular conjecture, see \cite{Lu89}) is known to hold except possibly for some very small values of $\ell$, see \cite{Lu90}, \cite{KT1} and \cite{KT2}. Nevertheless, we shall assume $\ell$ to be at least $h$. This ensures the existence of $\ell$-regular weights in $X$ and and we shall proceed just as \cite{A86}. 
 
 \subsection{KL-polynomials and their $\mu$-coefficients}
 Let $\hat W$ be the affine Weyl group associated with $R$. This group is generated by the reflections $(s_\alpha)_{\alpha \in S}$ together with the affine reflection $s_{\alpha_0, 1}$ given by  $s_{\alpha_0,1} \cdot \lambda = s_{\alpha_0}\cdot \lambda + \ell \alpha_0, \lambda \in X$ (in the modular case we replace $\ell$ by $p$). 
  
 For any two pairs $y,w \in \hat W$ we denote by $P_{y,w} \in \Z[t]$ the KL polynomial associated to $y,w$, see \cite{KL}.  Then $P_{y,w} = 0$ unless $y \leq w$ (in the Bruhat order on $\hat W$) and we have $P_{w,w} = 1$ for all $w$. The degree of $P_{y,w}$ is at most $(\ell(w) -\ell(y)-1)/2$ for all $y < w$.
 
\begin{defn} Let $y,w \in \hat W$ satisfy $y < w$. Then we set  $\mu(y,w)$ equal to the coefficient of $t^{(\ell(w) - \ell(y) - 1)/2)}$ in $P_{y,w}$.
\end{defn}

Observe that $\mu(y,w)$ is zero when the lengths of $y$ and $w$ are both even or both odd. It will turn out that $\mu(y,w)$ may well be zero also in many cases where the difference $\ell (w) - \ell(y)$ is odd. 

Recall that KL-polynomials are given by an explicit algorithm, see \cite{KL}.

\subsection{The relation between extensions of simple modules and the $\mu$-coefficients}
  Here we assume $\ell \geq h$.
 
 Recall that $A^+ \subset X$ denotes the bottom alcove in $X^+$.

 Our assumption on $\ell$ ensures that $A^+ \neq \emptyset$. If $B$ is another alcove in $X$ then $B = w\cdot A^+$ for a unique $w \in \hat W$. 

Set $A^- = w_0\cdot A^+$ and fix a weight $\lambda \in A^-$. Then we have
 
 \begin{thm}. Let $y, w \in \hat W$ with $y < w$ and $y\cdot \lambda, w \cdot \lambda \in X^+$  Choose a lower wall of the alcove containing $w \cdot \lambda$ and choose a reflection $s$ in a wall of this alcove with $sw\cdot \lambda < w\cdot \lambda$. If also $sy\cdot \lambda < y\cdot \lambda$ then we have
 $$ \dim_{\C} \Ext_{\mathcal C_q}^1(L_q(sw\cdot \lambda),  L_q(y\cdot \lambda)) = \mu (sw, y).$$
 \end{thm}
 
 \begin{proof} We extract this result from Section 2 in \cite{A86}. Note that we have chosen notation and assumptions in accordance with that reference.
 
 By \cite{A86}, Proposition 2.8 we have 
 $$\dim_{\C} \Ext^1_{\mathcal C_q}(L_q(sw\cdot \lambda),L_q(y.\lambda))= \dim_{\C} \Ext^1_{\mathcal C_q}(L_q(sw\cdot \lambda), H^0_q(y.\lambda))$$. Comparing this with \cite{A86}, Proposition 2.12 we get the equality in the theorem. 
 \end{proof}
 
 \begin{rem} 
 \begin {enumerate}
 \item This theorem determines the $\Ext$-groups between simple modules with highest weights in the orbit of $\lambda \in A^-$. However, the translation principle (see Section II.7 in\cite{RAG})
  then ensures that the answer is the same for any other weight $\nu \in A^-$.
 \item As we proved in Section 4 we only have to determine finitely many such extensions of simple modules, e.g. those corresponding to weights in the region $\{\nu \in X^+ \mid \langle \nu + \rho, \alpha_0^\vee \rangle \leq 3 \ell(h-1 )\}$. Then the remaining ones follow via the decompositions of tensor products of simple modules in $\mathcal C_0$.
 \end{enumerate}
 \end{rem}

\vskip 1 cm
\end{document}